\documentclass[12pt]{article}
\usepackage{amsmath}
\usepackage{amssymb}
\usepackage{mathrsfs}
\usepackage{amsfonts}
\usepackage{multirow}
\usepackage{cases}
\usepackage{graphicx}
\usepackage{subfigure}

\allowdisplaybreaks[4]

\newtheorem{Theorem}{Theorem}[section]
\newtheorem{Definition}{Definition}[section]
\newtheorem{Proposition}{Proposition}[section]
\newtheorem{Lemma}{Lemma}[section]

\newtheorem{Conjecture}{Conjecture}[section]

\newenvironment{proof}{{\textbf{Proof.}}\,}{\hfill$\hbox{\rule{5pt}{5pt}}$\\}

\newcommand{\Section}[1]{
        \par
        \stepcounter{section}
        \settowidth{\hangindent}{\large\bf\thesection.~}
        \hangafter=1
        \bigskip\bigskip\noindent
        {\large\bf\hbox{\thesection.~}#1}\par
        \nopagebreak
        \medskip
}
\newcommand{\alglist}{
\begin{list}{Step 1}
{\setlength{\leftmargin}{1.1 in}\setlength{\labelwidth}{1.0 in}} }

\begin{document}
\title{Computing all Laplacian H-eigenvalues for a $k$-uniform loose path of length three\thanks{This
work was supported by the National Natural Science Foundation of
China (Grant No. 11771244).}}

\author{Junjie Yue  \quad Liping Zhang \thanks
{Corresponding author (lzhang@math.tsinghua.edu.cn).} \\
{\small  Department of Mathematical Sciences, Tsinghua University, Beijing 100084, China}}

\date{}
\maketitle

{\bf Abstract.} The spectral theory of Laplacian tensor  is an important tool for
revealing some important properties of a hypergraph. It is meaningful to compute all Laplacian H-eigenvalues for some special $k$-uniform hypergraphs. For an odd-uniform loose path of length three, the Laplacian H-spectrum  has been studied. However, all Laplacian H-eigenvalues of the class of loose paths have not been found out.  In this paper, we compute all Laplacian H-eigenvalues for the class of loose paths. We show that the number of Laplacian H-eigenvalues of an odd(even)-uniform loose path with length three is $7$($14$). Some numerical results are given to show the efficiency of our method.  Especially, the numerical results show that its Laplacian H-spectrum  converges to $\{0,1,1.5,2\}$ when $k$ goes to infinity. Finally, we establish convergence analysis for a part of the conclusion and also present a conjecture.

 \vspace{2mm}



\vspace{2mm}

%\newpage

\Section{Introduction}

A natural definition for the Laplacian tensor and the signless Laplacian tensor of a $k$-uniform hypergraph for $k\ge 3$ was introduced in \cite{eig3}. From then on,
the study of spectral hypergraph theory via tensors has
attracted extensive attention and interest \cite{Hyp1,CPZ1,Hu1,Hyp2,Hyp3,LiG,weak2,eig2,eig3,qibook,eig23,xie1}. Some results extended
the classical results in spectral graph theory \cite{unf2}. Xie and  Chang \cite{xie1} discussed the applications of  the largest and the smallest H-eigenvalues of Laplacian tensors and signless Laplacian tensors in the edge cut and the edge connectivity of a $k$-uniform hypergraph. Qi \cite{eig3} introduced $H^+$-eigenvalues of Laplacian tensors and signless Laplacian tensors and established some  properties of these$H^+$-eigenvalues. In \cite{eig3}, Qi also defined  analytic connectivity of a uniform hypergraph and discussed its application in edge connectivity.  Hu, Qi and Xie \cite{Hyp2} studied the largest Laplacian and signless Laplacian eigenvalues of a $k$-uniform hypergraph, and generalized some classical results of spectral graph theory to spectral hypergraph theory, in particular, when $k$ is even. They showed that the largest Laplacian
$H$-eigenvalue of a connected $k$-uniform hypergraph  is equal to
its largest signless Laplacian $H$-eigenvalue if and only if it is
odd-bipartite \cite{Hyp2}. A $k$-uniform hypergraph is called odd-bipartite if $k$ is even, and the vertex set of the hypergraph can be divided to two parts, such that each edge has odd number of vertices in each of these two parts.

Recently, Qi, Shao and Wang \cite{eig23} shown that the largest
signless Laplacian $H$-eigenvalue of a connected $k$-uniform
hypergraph $G$, reaches its upper bound $2\Delta$, where $\Delta$ is
the largest degree of $G$, if and only if $G$ is regular, and that
the largest Laplacian $H$-eigenvalue of $G$, reaches the same upper
bound, if and only if $G$ is regular and odd-bipartite. What kind of
$k$-uniform hypergraph $G$ is regular and odd-bipartite? To answer
this question, they introduced
$s$-paths and $s$-cycles and studied their properties. Clearly, when
$s=1$, $s$-path extended the path in an ordinary graph, which is
called loose path \cite{CPZ1,loosepath}. They pointed out that $s$-path cannot
be regular but is odd-bipartite when $k\ge 4$
\cite[Proposition 4.1]{eig23}. Hu, Qi and Shao \cite{CPZ1} introduced the class of cored hypergraphs and power hypergraphs, and investigated the properties of their Laplacian H-eigenvalues. Power hypergraphs are cored hypergraphs, but not vice versa. They showed that loose paths are power hypergraphs, while $s$-paths and for $2\le s< \frac{k}{2}$ are cored hypergraphs, but not power hypergraphs in general. Moreover, it is shown that  the largest Laplacian H-eigenvalue of an even-uniform cored hypergraph is equal to its largest signless Laplacian H-eigenvalue. Especially, they found out  the Laplacian H-spectra of the $k$-uniform loose path of
length $3$ in \cite[Proposition 5.4]{CPZ1} when $k$ is odd.
 However, \cite[Proposition 5.4]{CPZ1} can not compute out all H-eigenvalues
of its Laplacian tensor. Very recently, Chang, Chen and Qi \cite{Changjy} proposed an efficient first-order optimization algorithm  for computing extreme
H- and Z-eigenvalues of sparse tensors arising from large scale uniform hypergraphs.

These results raised several questions. Firstly, can we identify the
largest Laplacian and signless Lapacian $H$-eigenvalues of
$k$-uniform loose paths? This question was studied by Yue, Zhang and Lu \cite{hyper2}. For $k$-uniform loose paths, they showed in \cite{hyper2} that the largest H-eigenvalues of their adjacency tensors, Laplacian tensors, and
signless Laplacian tensors are computable. Secondly, can we identify the
largest adjacency and signless Lapacian $H$-eigenvalues of
power hypergraphs and cored hypergraphs? This question was discussed by Yue, Zhang, Lu and Qi \cite{hyper1}. In \cite{hyper1},  they  studied the adjacency and signless Laplacian tensors of cored hypergraphs and power hypergraphs. They investigated the properties of their adjacency
and signless Laplacian H-eigenvalues. Especially, they found out the largest H-eigenvalues
of adjacency and signless Laplacian tensors for uniform squids. Moreover, they also computed the
H-spectra of sunflowers.  Thirdly, can we calculate all Laplacian H-eigenvalues for some special $k$-uniform hypergraphs,
such as loose paths and loose cycles? This is useful if one wishes to study the second smallest Laplacian H-eigenvalue  of
a $k$-uniform hypergraph, as the second smallest Laplacian eigenvalue of a graph plays a key role in spectral graph theory \cite{unf2,unf3,CPZ1}.

%Moreover, Chen \cite{chenyannan} found that the second smallest Laplacian H-eigenvalue of a $k$-uniform hypergraph has important applications in image processing.

Motivated by the third question and the above mentioned applications.  In this paper, we pay our attention on
the Laplacian H-spectrum of $k$-uniform loose paths of length three. We observe that \cite[Proposition 5.4]{CPZ1} can not find out all  Laplacian H-eigenvalues
of the odd-uniform loose path of length three. There exists another case except four cases  in \cite[Proposition 5.4]{CPZ1}. In addition, there is no results when $k$ is even.  Using the same method as \cite{CPZ1}, we compute
the Laplacian H-spectrum of the loose path of length three. This will be useful for research on the second smallest Laplacian  H-eigenvalue of a $k$-uniform hypergraph, and also will be useful for research on applications in the edge cut and the edge connectivity of a $k$-uniform hypergraph. Especially, we show that when $k$ is odd, the number of Laplacian H-eigenvalues
is $7$, and when $k$ is even, the number of Laplacian H-eigenvalues is $14$. Some numerical results are given to show the efficiency of our method.  Especially, the numerical results show that its Laplacian H-spectrum  converges to $\{0,1,1.5,2\}$ when $k$ goes to infinity. Finally, we establish convergence analysis for a part of the conclusion and also present a conjecture.

The rest of this paper is organized as follows. We list some known
results of cored hypergraphs and power hypergraphs in the next
section. In Section 3, we investigate the number and distribution of
Laplacian H-eigenvalues of the $k$-uniform loose path with length three.
In Section 4, numerical experiments are implemented  and some convergence analyses are established. Finally, some concluding remarks are given in Section 5.

\Section{Preliminaries}
Some known results about cored hypergraphs and power hypergraphs are given in this section, which will be used in the sequel. Denote $[n]:=\{1,\ldots, n\}$.  A real tensor $\mathcal{T}=(t_{i_1 \cdots
i_k})$ of order $k$ and dimension $n$ refers to a multidimensional
array or a hypermatrix with entries $t_{i_1 \cdots i_k}$ such that
$t_{i_1 \cdots i_k}\in \mathbb{R}$ for all $i_j\in [n]$ and $j\in[k]$.
Given a vector $x\in \mathbb{R}^n$, $\mathcal{T} x^{k-1}$ is defined as an
n-dimensional vector such that its $i$th element being
$\sum\limits_{i_2,\cdots,i_k\in [n]} t_{ii_2 \cdots i_k} x_{i_2}
\cdots x_{i_k}$ for $i\in [n]$. Let $\mathcal{I}$ be the identity
tensor of appropriate dimension, e.g., $i_{i_1\cdots i_k}=1$ if and
only if $i_1=\cdots=i_k\in[n]$, and zero otherwise when the
dimension is $n$. The following definition was introduced in
\cite{eig2}.

\begin{Definition} \label{defeig}
{\rm Let $\mathcal{T}$ be a $k$-th order $n$-dimensional real tensor. For
some $\lambda\in\mathbb{R}$, if polynomial system $(\lambda
\mathcal{I}-\mathcal{T})x^{k-1}=0$ has a solution
$x\in\mathbb{R}^n\backslash \{0\}$, then $\lambda$ is called an
H-eigenvalue and $x$ an H-eigenvector.}
\end{Definition}

Obviously, H-eigenvalues are real number. The number of
H-eigenvalues of a real tensor is finite \cite{eig1,eig2}. By
\cite{eig3}, all the tensors given in Definition \ref{defLQ} have at
least one H-eigenvalue \cite{weak3}. Hence, we can denote by
$\lambda(\mathcal{T})$ as the largest H-eigenvalue of a real tensor
$\mathcal{T}$.

We next list some essential notions of uniform hypergraphs. Please refer to
\cite{unf1,unf2,unf3,GR1,weak1,eig3} for comprehensive references.
In this paper, unless stated otherwise, a hypergraph means an
undirected simple $k$-uniform hypergraph $G$ with vertex set $V$ and
edge set $E$. For a subset $S\subset[n]$, we denote by $E_S$ the set
of edges $\{e\in E| S\cap e\neq \emptyset\}$. For a vertex $i\in V$,
we simplify $E_{\{i\}}$ as $E_i$. It is the set of edges containing
the vertex $i$, i.e., $E_i:=\{e\in E|i\in e\}$. The cardinality
$|E_i|$ of the set $E_i$ is defined as the degree of the vertex $i$,
which is denoted by $d_i$. Two different vertices $i$ and $j$ are
connected to each other (or the pair $i$ and $j$ is connected), if
there is a sequence of edges $(e_1,\cdots,e_m)$ such that $i\in
e_1$, $j\in e_m$ and $e_r\cap e_{r+1}\neq \emptyset$ for all
$r\in[m-1]$. A hypergraph is called connected, if every pair of
different vertices of $G$ is connected. In the sequel, unless stated
otherwise, all the notations introduced above are reserved for the
specific meanings. For the sake of simplicity, we mainly consider
connected hypergraphs in the subsequent analysis because the
conclusion on connected hypergraphs can be easily generalized to
general hypergraphs via the techniques in \cite{weak1,eig3}.

Qi \cite{eig3} introduced the following definition for the
adjacency, Laplacian and signless Lapacian tensors of $k$-uniform
hypergraphs.

\begin{Definition} \label{defLQ}
{\rm Let $G=(V,E)$ be a $k$-uniform hypergraph. The adjacency tensor of
$G$ is defined as the $k$-th order $n$ dimensional tensor
$\mathcal{A}$ whose $(i_1\cdots i_k)$-entry is:
\begin{equation*}
a_{i_1\cdots i_k}:=
\begin{cases}
   \frac{1}{(k-1)!} &\mbox{if $\{i_1,\cdots,i_k\}\in E$ },\\
   0 &\mbox{otherwise. }
   \end{cases}
\end{equation*}
Let $\mathcal{D}$ be a $k$-th order $n$-dimensional diagonal tensor
with its diagonal element $d_{i\cdots i}$ being $d_i$, the degree of
vertex $i$, for all $i\in [n]$. Then
$\mathcal{L}:=\mathcal{D}-\mathcal{A}$ is the Laplacian tensor of
the hypergraph $G$, and $\mathcal{Q}:=\mathcal{D}+\mathcal{A}$ is
the signless Laplacian tensor of the hypergraph $G$.}
\end{Definition}

It is known that zero is always the smallest H-eigenvalue of
$\mathcal{L}$, $d\leq\lambda(\mathcal{L})\leq
\lambda(\mathcal{Q})\leq 2d$, where $d$ is the maximum degree of $G$
\cite{eig3}, and $\bar d\leq\lambda(\mathcal{A})\leq d$
\cite[Theorem 3.8]{Hyp1}, where $\bar d$ is the average degree of
$G$.

In the following, we recall the
concept of loose path introduced in
\cite{CPZ1,loosepath}. Some known results on $k$-uniform loose paths are listed here.

\begin{Definition} \label{def2.1}
{\rm Let $G=(V,E)$ be a $k$-uniform hypergraph. If we can number the vertex set $V$ as $V:=\{i_{1,1},\cdots,i_{1,k},i_{2,2},\cdots,i_{2,k},\cdots,i_{d-1,k},i_{d,2},\cdots,i_{d,k}\}$ for some positive integer $d$ such that $E=\{\{i_{1,1},\cdots,i_{1,k} \},\{i_{1,k},i_{2,2},\cdots,i_{2,k} \},\cdots,\{i_{d-1,k},i_{d,2},\cdots,i_{d,k}\}\}$, then $G$ is a loose path. $d$ is the length of the loose path.}
\end{Definition}
\begin{figure}[!htb]
\begin{center}
\includegraphics[height=0.4in,width=2.4in]{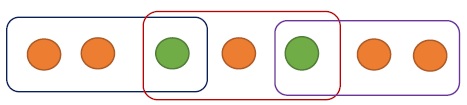}
\caption{\small An example of a $3$-uniform loose path of length $3$}
\label{fig2}
\end{center}
\end{figure}
\begin{figure}[!htb]
\begin{center}
\includegraphics[height=0.4in,width=3.2in]{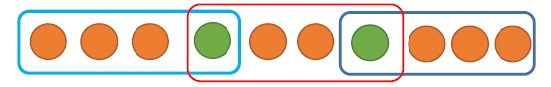}
\caption{\small An example of a $4$-uniform loose path of length $3$}
\label{fig3}
\end{center}
\end{figure}

 Let $G$ be a $k$-uniform loose path, we
have $\lambda(\mathcal{L})=2$ when $k$ is odd \cite{CPZ1,Hyp3}, and $2<\lambda(\mathcal{L})<3$ when $k$ is even \cite{hyper2}. In \cite[Proposition 5.4]{CPZ1}, the H-spectrum of Laplacian tensor was characterized for an odd-uniform loose path with length three. However, \cite[Proposition 5.4]{CPZ1} can not compute out all Laplacian H-eigenvalues.  Can we characterize the Laplacian H-spectrum for a $k$-uniform loose path of length three? In the next section, we will investigate this question.

The next two propositions follow from \cite[Corollary 1]{hyper2} and \cite[Proposition4.2, Corollary 5.1]{CPZ1} respectively.
\begin{Proposition}\label{largeL}
{\rm Let $G=(V,E)$ be a $k$-uniform loose path with length $d\ge 3$ and $\mathcal{L}$ be its Laplacian tensor. Then $\lambda(\mathcal{L})=2$ if $k$ is odd, and $2<\lambda(\mathcal{L})<3$ if $k$ is even.}
\end{Proposition}
\begin{Proposition}\label{lambda1}{\rm
Let $G=(V,E)$ be a $k$-uniform loose path with length $d\ge 3$ and $\mathcal{L}$ be its Laplacian tensor. Then $\lambda=1$ is an H-eigenvalue of $\mathcal{L}$.}
\end{Proposition}
The following proposition was given in \cite[Proposition 5.4]{CPZ1}.
\begin{Proposition}\label{pro5.4}{\rm
Let $k$ be odd and $G=(V,E)$ be a $k$-uniform loose path with length $d=3$. Let $\mathcal{L}$ be its Laplacian tensor.
Then $\lambda\neq 1$ is an H-eigenvalue of $\mathcal{L}$ if and only if one of the following four cases happens:

(i)$\lambda =2$ or $\lambda=0$,

(ii)$\lambda$ is the unique root of the equation $(\lambda-2)(1-\lambda)^{k-1}+1=0$, which is in $(0,1)$,

(iii)$\lambda$ is the unique root of the equation $(\lambda-2)^2 (1-\lambda)^{k-2}-1=0$, which is in $(0,1)$,

(iv)$\lambda$ is a real root of the equation $(\lambda-2)^2(1-\lambda)^{k-1}+ 2 \lambda -3=0$ in $(0,2)$.}
\end{Proposition}

\Section{Laplacian H-spectrum of the loose path with length three }

In this section, we discuss Laplacian tensor of a $k$-uniform loose path with length $d=3$. We
compute all H-eigenvalues of its Laplacian tensor. Especially, we show that the number of Laplacian H-eigenvalues is $7$ when $k$ is odd, and $14$ when $k$ is even.

Let $G_{k,3}$ denote  a $k$-uniform loose path with length $d=3$ throughout the rest of this paper. By Proposition \ref{lambda1},  $\lambda=1$ is  a Laplacian H-eigenvalue of $G_{k,3}$ \cite[Corollary 5.1]{CPZ1}. Let $\lambda$ be an H-eigenvalue of Laplacian tensor of $G_{k,3}$. By Proposition \ref{pro5.4}, when $k$ is odd, if $\lambda\neq 1$ then $\lambda$ lies in one of the four cases \cite[Proposition 5.4]{CPZ1}. Actually, we prove that there are five cases and the number of Laplacian H-eigenvalues is $7$. Our results are given in the following theorem.

\begin{Theorem}\label{l1em2.1}{\rm
Let $k$ be odd and $\mathcal{L}$ be  Laplacian tensor of $G_{k,3}$.
Then $\lambda\neq 1$ is an H-eigenvalue of $\mathcal{L}$ if and only if one of the following five cases happens:

(i) $\lambda =2$ or $\lambda=0$,

(ii) $\lambda$ is the unique root of the equation $(\lambda-2)(1-\lambda)^{k-1}+1=0$, which is in $(0,1)$,

(iii) $\lambda$ is the unique root of the equation $(\lambda-2)^2 (1-\lambda)^{k-2}-1=0$, which is in $(0,1)$,

(iv) $\lambda$ is a real root of the equation $(\lambda-2)^2(1-\lambda)^{k-1}+ 2 \lambda -3=0$. There are two real roots, one is in $(0,1)$ and the other is in $(1,\frac{2k}{k+1})$,

(v) $\lambda$ is a real root of the equation $[(\lambda-2)(\lambda-1)^{k-1}+1]^2-(1-\lambda)^k=0$, which is in $(0,1)$.}
\end{Theorem}
\begin{proof}
Suppose that $E=\{\{1,\cdots,k\},\{k,\cdots,2k-1\},\{2k-1,\cdots,3k-2\}\}$, and $x\in\Re^n$ be an H-eigenvector of $\mathcal{L}$ corresponding to the H-eigenvalue $\lambda \neq 1$. By \cite[Lemma 5.2]{CPZ1}, $\lambda\ge 0$.
By \cite[Lemmas 3.1 and 4.1]{CPZ1}, we have $x_1=\cdots=x_{k-1}=\frac{x_k}{1-\lambda}$ if there are nonzero; $x_{k+1}=\cdots=x_{2k-2}=\pm\sqrt{\frac{x_kx_{2k-1}}{1-\lambda}}$ if there are nonzero; and $x_{2k}=\cdots=x_{3k-2}= \frac{x_{2k-1}}{1-\lambda}$ if there are nonzero. By the analogous argument to the proof of \cite[Proposition 5.4]{CPZ1}, we can show (i), (ii) and (iii).

To show (iv) and (v),  according to the proof process of \cite[Proposition 5.4]{CPZ1}, we consider the case of $x_k\neq0$ and $x_{2k-1}\neq0$. We divide into two subcases to discuss as follows.

(a) If $x_{k+1}\neq0$, $x_1 \neq 0$ and $x_{2k} = 0$, we have
\begin{equation*}
\begin{cases}
   (\lambda-2) x^{k-1}_{k}=- (\pm \sqrt{\frac{x_k x_{2k-1}}{1-\lambda}})^{k-2} x_{2k-1},  \\
   (\lambda-2) x^{k-1}_{2k-1}=- (\pm \sqrt{\frac{x_k x_{2k-1}}{1-\lambda}})^{k-2} x_{k}-(\frac{x_{2k-1}}{1-\lambda})^{k-1}.
   \end{cases}
\end{equation*}
Let $t:=\frac{x_k}{x_{2k-1}}$. Then we obtain
\begin{equation*}
\begin{cases}
   (\lambda-2) (1-\lambda)^{k-1} t^k=(\lambda-2)(1-\lambda)^{k-1}+1,  \\
   (\lambda-2)^2 (1-\lambda)^{k-2} t^k=1.
   \end{cases}
\end{equation*}
The first equality shows $\lambda<2$. So, the above equation yields  $(\lambda-2)^2 (1-\lambda)^{k-1}+2\lambda -3=0$.
Let $$f(\lambda)=(\lambda-2)^2 (1-\lambda)^{k-1} -3,\quad g(\lambda)=-2\lambda.$$ Then $\lambda$ is a root of the equation $f(\lambda)-g(\lambda)=0$. We compute
$$f'(\lambda)=(\lambda-2)(\lambda-1)^{k-2}[(1+k)\lambda-2k]$$ and $$f''(\lambda)=(\lambda-1)^{k-3}[(k^2+k)\lambda^2 -4k^2 \lambda +4k^2-4k+2].$$
Note that $f'(\lambda)<0$ if $\frac{2k}{1+k}<\lambda<2$ or $0<\lambda<1$;  $f'(\lambda)>0$ if  $1<\lambda<\frac{2k}{1+k}$. It is easy to see that
$$f(0)=1,\quad f(1)=-3,\quad f(2)=-3,\quad f(\frac{2k}{1+k})=\frac{4}{(1+k)^2}\left(\frac{1-k}{1+k}\right)^{k-1}-3>-3,$$
and $$g(0)=0,\quad g(1)=-2,\quad g(2)=-4,\quad g(\frac{2k}{1+k})=-4+\frac{4}{k+1} \leq -3.$$
So, the equation $f(\lambda)-g(\lambda)=0$ has a root in $(0,1)$ and $(1,\frac{2k}{k+1})$, respectively. Clearly, there is a unique real root in $(1,\frac{2k}{k+1})$ since $f'(\lambda)-g'(\lambda)>0$ when $1<\lambda<\frac{2k}{k+1}$.
We now show that there is a unique root in $(0,1)$. Let $$h(\lambda)=(k^2+k)\lambda^2 -4k^2 \lambda +4k^2-4k+2.$$
 Then, $f''(\lambda)=(\lambda-1)^{k-3}h(\lambda)$. Since the minimum point of $h(\lambda)$ is  $\frac{2k}{k+1}=2-\frac{2}{k+1}>1$ and $h(1)=(k-2)(k-1))>0$, $f''(\lambda)> 0$ if $\lambda\in(0,1)$. Hence, $f'(\lambda)$ is monotone increasing in  $\lambda\in(0,1)$.  Combining $f'(0)=-4k<-2$ and $f'(1)=0>-2$, we obtain that the function $f(\lambda)-g(\lambda)$ is first decreasing and then increasing. It can be seen that  $f(0)-g(0)=1$ and $f(1)-g(1)=-1$. Thus, it has a unique root in $(0,1)$. Hence, in this subcase, we can show (iv).

 The discussion for the subcase  $x_{k+1}\neq0$, $x_1 = 0$ and $x_{2k} \neq 0$ is similar, and the result is the same as the above subcase.

(b) If $x_{k+1}\neq0$, $x_1 \neq 0$ and $x_{2k} \neq 0$, we have
\begin{equation}\label{3lodd4}
  \begin{cases}
    (\lambda-2) x^{k-1}_{k}=- (\pm \sqrt{\frac{x_k x_{2k-1}}{1-\lambda}})^{k-2} x_{2k-1}-(\frac{x_{k}}{1-\lambda})^{k-1},  \\
     (\lambda-2) x^{k-1}_{2k-1}=- (\pm \sqrt{\frac{x_k x_{2k-1}}{1-\lambda}})^{k-2} x_{k}-(\frac{x_{2k-1}}{1-\lambda})^{k-1}.
  \end{cases}
\end{equation}
Assume $\lambda>1$, we have $x_k x_{2k-1}<0$ and $(\lambda-2)(1-\lambda)^{k-1}+1=0$.  This is a contradiction with (ii). Thus, $\lambda<1$.

on one hand, if $x_k =x_{2k-1}$, by (\ref{3lodd4}), we have
\begin{equation*}
  [(\lambda-2)(1-\lambda)^{k-1}+1]^2-(1-\lambda)^k=0.
\end{equation*}
Let $$f(\lambda)=[(\lambda-2)(1-\lambda)^{k-1}+1]^2,\quad g(\lambda)=(1-\lambda)^k, \quad h(\lambda)=(\lambda-2)(1-\lambda)^{k-1}+1.$$ By a similar discussion as the proof of (ii) in \cite[Proposition 5.4]{CPZ1}, we know that $h(\lambda)$ is a monotonically  increasing function as $\lambda<1$. Hence, $f(\lambda)$ is first decreasing  and then increasing in $(0,1)$. Since $ g'(\lambda)=-k(\lambda-1)^{k-1}<0$ , $f(0)=1$ and $g(0)=1$, it has at least one root in $(t,1)$, where $t$ is the unique real root of  $(t-2)(1-t)^{k-1}+1=0$ in $(0,1)$.

Now we need to prove that if there is a real root in the interval $(0,t)$? The problem can be changed into the following minimization problem:
\begin{equation} \label{eq1}
\begin{array}{rcl}
 \min && T(\lambda)=(1-\lambda)^k-[(\lambda-2)(1-\lambda)^{k-1}+1]^2 \\
 {\rm s.t.} &&  -1\leq (\lambda-2)(1-\lambda)^{k-1}+1\leq 0.
\end{array}
\end{equation}
By the first order necessary condition of (\ref{eq1}), we have
\begin{equation}\label{mykkt}
\begin{cases}
   -k(1-\lambda)^{k-1}-2[(\lambda-2)(1-\lambda)^{k-1}+1](1-\lambda)^{k-2}(-k\lambda+2k-1)\\+(\mu_2-\mu_1)(1-\lambda)^{k-2}(-k\lambda+2k-1)=0,  \\\mu_1[(\lambda-2)(1-\lambda)^{k-1}+2] =0, \,\,
    \mu_1 \geq 0, \,\, (\lambda-2)(1-\lambda)^{k-1}+2\geq 0,\\
   \mu_2 [-(\lambda-2)(1-\lambda)^{k-1}-1] =0, \,\,
   \mu_2 \geq 0,\,\, -(\lambda-2)(1-\lambda)^{k-1}-1 \geq 0.
   \end{cases}
\end{equation}
By trivial analyzing and discussing (\ref{mykkt}) in details, we obtain that there is no a real root of $[(\lambda-2)(1-\lambda)^{k-1}+1]^2-(1-\lambda)^k=0$ in $(0,t)$.

On the other hand, if $x_k \neq x_{2k-1}$, we have $(\lambda-2)(1-\lambda)^{k-1}+1=0$. However, by the first equation of (\ref{3lodd4}),
$$(1-\lambda)^{k-1}(\pm \sqrt{\frac{x_k x_{2k-1}}{1-\lambda}})^{k-2} x_{2k-1}=0.$$ This is a contradiction. Thus, we complete the proof.
\end{proof}

We next  characterize all Laplacian H-eigenvalues of $G_{k,3}$ when $k$ is even. We compute out  all H-eigenvalues of its Laplacian tensor. Let $\lambda$ be an H-eigenvalue of Laplacian tensor of $G_{k,3}$ where $k$ is even. We show that if $\lambda\neq 1$ then $\lambda$ lies in one of seven positions, and the number of Laplacian H-eigenvalues is $14$.

\begin{Theorem}\label{l1em2.3}{\rm
Let $k$ be even and $\mathcal{L}$ be  Laplacian  tensor of $G_{k,3}$.
Then $\lambda\neq 1$ is an H-eigenvalue of $\mathcal{L}$ if and only if one of the following seven cases happens:

(i)$\lambda =0$ or $\lambda =2$.

(ii)$\lambda$ is a real root of the equation $(\lambda-2)(\lambda-1)^{k-1}-1=0$; there are two roots in $(2,3)$ and $(0,1)$, respectively.

(iii)$\lambda$ is a real root of the equation $(\lambda-2)^2 (\lambda-1)^{k-2}-1=0$; there are two roots in $(2,3)$ and $(0,1)$, respectively.

(iv)$\lambda$ is a real root of the equation $(\lambda-2)^2(\lambda-1)^{k-1}- 2 \lambda +3=0$; there are three roots in $(2,3)$, $(1,\frac{2k}{1+k})$ and $(0,1)$, respectively.

(v)$\lambda$ is the unique root of the equation $(\lambda-2)^2(1-\lambda)^{k-1}+1=0$, which is in $(2,3)$.

(vi)$\lambda$ is a real root of the equation $[(\lambda-2)(\lambda-1)^{k-1}-1]^2-(\lambda-1)^k=0$; there are two roots in $(2,3)$ and $(0,1)$, respectively.

(vii)$\lambda$ is the unique real root of the equation $[(\lambda-2)(\lambda-1)^{k-1}+1]^2-(1-\lambda)^k=0$, which is in $(2,3)$.}
\end{Theorem}
\begin{proof}
Suppose that $E=\{\{1,\cdots,k\},\{k,\cdots,2k-1\},\{2k-1,\cdots,3k-2\}\}$, and $x\in\Re^n$ be an H-eigenvector of $\mathcal{L}$ corresponding to $\lambda \neq 1$.

By \cite[Lemma 3.1]{CPZ1} and \cite[Lemma 2.6]{hyper2}, we have $|x_1|=\cdots=|x_{k-1}|=|\frac{x_k}{1-\lambda}|$ if there are nonzero; $|x_{k+1}|=\cdots=|x_{2k-2}|=\sqrt{|\frac{x_kx_{2k-1}}{1-\lambda}|}$ if there are nonzero; and $|x_{2k}|=\cdots=|x_{3k-2}|= |\frac{x_{2k-1}}{1-\lambda}|$ if there are nonzero.

The proof is divided into four cases, which contain several subcases respectively.

{\bf Case 1:} We assume $x_k=0$ and $x_{2k-1}=0$. By \cite[Lemma 2.6]{hyper2}, we have $x_1=0$, $x_{k+1}=0$ and $x_{2k}=0$. So, this case does not happen.

{\bf Case 2:} We assume $x_k=0$ and $x_{2k-1}\neq0$. Then, $x_1=0$ and $x_{k+1}=0$.

(a) If $x_{2k}=0$, then $\lambda =2$ is an H-eigenvalue.

(b) Let $p+q=k-1$. If $x_{2k} \neq 0$,  we have
\begin{equation*}
  (\lambda-2) x^{k-1}_{2k-1}= - (\frac{x_{2k-1}}{\lambda-1})^p  (\frac{x_{2k-1}}{1-\lambda})^q.
\end{equation*}

(b1) If $p$ is odd, we have
\begin{equation*}
  (\lambda-2) x^{k-1}_{2k-1}= - (\frac{x_{2k-1}}{\lambda-1})^{k-1}  = (\frac{x_{2k-1}}{1-\lambda})^{k-1},
\end{equation*}
which deduces
\begin{equation}\label{eqzlp1}
  (\lambda-2)(1-\lambda)^{k-1} -1=0.
\end{equation}
Since $(\lambda-2)(1-\lambda)^{k-1}<0$ when $\lambda<1$ or $\lambda>2$, the equation (\ref{eqzlp1}) has no any real root under this condition.

(b2) If $p$ is even, we have
\begin{equation*}
  (\lambda-2) x^{k-1}_{2k-1}= - (\frac{x_{2k-1}}{1-\lambda})^{k-1}  = (\frac{x_{2k-1}}{\lambda-1})^{k-1},
\end{equation*}
which implies
\begin{equation*}
  (\lambda-2)(\lambda-1)^{k-1} -1=0.
\end{equation*}
Let $$f(\lambda)=  (\lambda-2)(\lambda-1)^{k-1} -1.$$ We compute
$$f'(\lambda)=(\lambda-1)^{k-2}(k \lambda -2k+1).$$ Clearly, $f'(\lambda)>0$ if $\lambda > \frac{2k-1}{k}$ and $f'(\lambda)<0$ if $\lambda<\frac{2k-1}{k}$. It is easy to see that $$f(0)=1,\quad f(1)=-1,\quad f(2)=-1,\quad f(3)=2^{k-1}-1>0.$$ Thus, the equation $f(\lambda)=0$ has two real roots respectively in  $(0,1)$ and $(2,3)$. So, (ii) holds.

{\bf Case 3:} We assume that $x_k\neq0$ and $x_{2k-1}=0$. After a discussion similar to Case 2, we can show that either $\lambda=2$ or (ii) also holds in this case.

{\bf Case 4:} We assume that $x_k\neq0$ and $x_{2k-1}\neq0$.

(a)  If $x_{k+1}=0$, $x_1=0$ and $x_{2k}=0$, then $\lambda =2$ is an H-eigenvalue.

(b) If $x_{k+1}=0$, $x_1\neq 0$ and $x_{2k}=0$, we have
\begin{equation*}
  \begin{cases}
    (\lambda-2) x^{k-1}_{k}= - (\frac{x_{k}}{\lambda-1})^p  (\frac{x_{k}}{1-\lambda})^q, \quad p+q=k-1,\\
    (\lambda-2) x^{k-1}_{2k-1}= 0,
  \end{cases}
\end{equation*}
which implies that this case does not happen. Similarly, the case $x_{k+1}=0$, $x_1 = 0$ and $x_{2k} \neq 0$ does not exist.

(c) If $x_{k+1}=0$, $x_1\neq 0$ and $x_{2k}\neq0$, then we have $(\lambda-2)(\lambda-1)^{k-1} -1=0$. By (b2), $\lambda$ is its root in the interval $(0,1)$ or $(2,3)$. That is, (ii) holds.

(d) If $x_{k+1}\neq0$, $x_1 = 0$ and $x_{2k} = 0$, we have
\begin{equation*}
  \begin{cases}
    (\lambda-2) x^{k-1}_{k}= -x_{2k-1} (\pm\sqrt{|\frac{x_k x_{2k-1}}{\lambda-1}|})^{k-2}, \\
    (\lambda-2) x^{k-1}_{2k-1}= -x_{k} (\pm\sqrt{|\frac{x_k x_{2k-1}}{\lambda-1}|})^{k-2},
  \end{cases}
\end{equation*}
which deduces
\begin{equation*}
  (\lambda-2) (x^{k}_{k}-x^{k}_{2k-1})=0.
\end{equation*}
This shows $\lambda =2$ or $ x_k=\pm x_{2k-1}$. So, In the second case, $\lambda$ is a real root of the equation
\begin{equation*}
  (\lambda-2)^2 (\lambda-1)^{k-2} -1=0.
\end{equation*}
Let $$f(\lambda)= (\lambda-2)^2 (\lambda-1)^{k-2} -1.$$
By a straightforward computation, we have
$$f'(\lambda)= (\lambda-2) (\lambda-1)^{k-3} (k\lambda -2k+2),$$ which implies that  $f'(\lambda)< 0$ if $\lambda<1$ or $\frac{2k-2}{k}<\lambda<2$ and $f'(\lambda)> 0$ if $1<\lambda<\frac{2k-2}{k}$ or $\lambda>2$.
Clearly, $f(0)=3$,  $f(1)=-1$,  $f(2)=-1$ and $f(3)=2^{k-2}-1>0$. We also have $f(\lambda)<0$ when $1<\lambda<2$. Hence, the equation $f(\lambda)=0$ has two real roots. One is in $(0,1)$ and the other is in $(2,3)$. Thus, (iii) holds.

(e) If $x_{k+1}\neq0$, $x_1 \neq 0$ and $x_{2k} =0$, we have
\begin{equation*}
  \begin{cases}
    (\lambda-2) x^{k-1}_{k}= -x_{2k-1}(\sqrt{|\frac{x_k x_{2k-1}}{\lambda-1}|})^{p_1} (-\sqrt{|\frac{x_k x_{2k-1}}{\lambda-1}|})^{q_1}-(\frac{x_{k}}{\lambda-1})^{p_2} (\frac{x_{k}}{1-\lambda})^{q_2}, \\
    (\lambda-2) x^{k-1}_{2k-1}= -x_{k}(\sqrt{|\frac{x_k x_{2k-1}}{\lambda-1}|})^{p_1} (-\sqrt{|\frac{x_k x_{2k-1}}{\lambda-1}|})^{q_1},\\
    p_1+q_1=k-2, \quad  p_2+q_2=k-1.
  \end{cases}
\end{equation*}
Let $t:=\frac{x_{2k-1}}{x_k}$. We consider the following two subcases.

(e1) If $p_2$ is odd, we have
\begin{equation*}
  \begin{cases}
    (\lambda-2) (\lambda-1)^{k-1} t^k=(\lambda-2)(\lambda-1)^{k-1}+1, \\
    (\lambda-2)^2 (\lambda-1)^{k-2} t^k=1.
  \end{cases}
\end{equation*}
This implies that $\lambda$ is a real root of the equation $ (\lambda-2)^2 (1-\lambda)^{k-1}+1=0$.

Let $f(\lambda)=(\lambda-2)^2 (1-\lambda)^{k-1}+1$, we have $$
f'(\lambda)=(\lambda-2)(1-\lambda)^{k-2}[-(1+k)\lambda+2k].$$
Clearly, $f'(\lambda)<0$ if $\lambda>2$. Also, $f(2)=1>0$ and $f(3)=(-2)^{k-1}+1<0$. Hence, the equation $f(\lambda)=0$ has a unique root in $(2,3)$. That is, (v) holds.

(e2) If $p_2$ is even, we have
\begin{equation*}
  \begin{cases}
    (\lambda-2) (\lambda-1)^{k-1} t^k=(\lambda-2)(\lambda-1)^{k-1}-1, \\
    (\lambda-2)^2 (\lambda-1)^{k-2} t^k=1.
  \end{cases}
\end{equation*}
This shows that $\lambda$ is a real root of the equation $$(\lambda-2)^2 (\lambda-1)^{k-1}-2\lambda +3=0.$$
 Let $$f(\lambda)=(\lambda-2)^2 (\lambda-1)^{k-1} +3,\quad g(\lambda)=2\lambda.$$
By a straightforward computation, we have the derivatives
 $$f'(\lambda)=(\lambda-2)(\lambda-1)^{k-2}[(1+k)\lambda-2k]$$ and$$f''(\lambda)=(\lambda-1)^{k-3}[(k^2+k)\lambda^2 -4k^2 \lambda +4k^2-4k+2].$$
Also, $f(0)=-1$, $f(1)=3$, $f(2)=3$, $f(3)=2^{k-1}+3\geq 11$ and $$f(\frac{2k}{1+k})=(\frac{2k}{1+k}-2)^2 (\frac{2k}{1+k}-1)^{k-1} +3 \leq max\{ (2-\lambda)(\lambda-1)+3 | 1<\lambda<2 \}=\frac{13}{4}.$$  Moreover, $g(0)=0$, $g(1)=2$, $g(2)=4$ and $g(3)=6$. If $k=4$, then $\frac{2k}{1+k}=\frac{8}{5}$ and
 $g(\frac{8}{5})=\frac{16}{5}>f(\frac{8}{5})$. If $k\ge 6$, we have
 $$g(\frac{2k}{1+k})=4-\frac{4}{k+1}\ge \frac{10}{3}>f(\frac{2k}{1+k}).$$
 Hence, the root of the equation $f(\lambda)-g(\lambda)=0$ exists in $(0,1)$, $(1,\frac{2k}{1+k})$ and $(2,3)$, respectively. After a similar discussion to the subcase (a) in the proof of Theorem \ref{l1em2.1},  we show that they are unique in each interval.

 In addition, if $x_{k+1}\neq0$, $x_1 = 0$ and $x_{2k} \neq 0$, after a discussion similar to Case 4 (e), we also have (iv).

(f) If $x_{k+1}\neq0$, $x_1\neq 0$ and $x_{2k}\neq 0$, we have
\begin{equation*}
  \begin{cases}
     (\lambda-2) x^{k-1}_{k}= -x_{2k-1}( \sqrt{|\frac{x_k x_{2k-1}}{\lambda-1}|})^{p_1} (-\sqrt{|\frac{x_k x_{2k-1}}{\lambda-1}|})^{q_1}-(\frac{x_{k}}{\lambda-1})^{p_2 } (\frac{x_{k}}{1-\lambda})^{q_2}, \\
    (\lambda-2) x^{k-1}_{2k-1}=-x_{k}(\sqrt{|\frac{x_k x_{2k-1}|}{\lambda-1}})^{p_1} (-\sqrt{|\frac{x_k x_{2k-1}}{\lambda-1}|})^{q_1}-(\frac{x_{2k-1}}{\lambda-1})^{p_3}(\frac{x_{2k-1}}{1-\lambda})^{q_3},\\
    p_1+q_1=k-2,\quad p_2+q_2=k-1,\quad p_3+q_3=k-1.
  \end{cases}
\end{equation*}

(f1) If $p_2$ and $p_3$ are odd, we have
\begin{equation*}
  [(\lambda-2)-\frac{1}{(1-\lambda)^{k-1}}](x^k_k-x^k_{2k-1})=0.
\end{equation*}
Since $(\lambda-2)(1-\lambda)^{k-1}-1\neq 0$, we get $x_k =\pm x_{2k-1}$ and
\begin{equation*}
  [(\lambda-2)(1-\lambda)^{k-1}-1]^2-(1-\lambda)^k=0.
\end{equation*}
Let $$
h(\lambda)=(\lambda-2)(1-\lambda)^{k-1}-1,\quad f(\lambda)=h(\lambda)^2,\quad g(\lambda)=(1-\lambda)^k.$$
We compute
$$f'(\lambda)=2h(\lambda)(1-\lambda)^{k-2}(-\lambda k+2k-1),\quad g'(\lambda)=k(\lambda-1)^{k-1}.$$
Clearly, $f(3)-g(3)=2^{2k-2}+1>0$. If $\lambda\geq3$, we have
$$f'(\lambda)-g'(\lambda)\geq (2^{k}+2)(\lambda k-2k+1)-k(\lambda-1)\geq 17\lambda k -35k +18\geq 16k+18>0.$$ Hence, the equation $f(\lambda)-g(\lambda)=0$ has no real root when $\lambda \ge 3$.
If $\lambda \in (1,2)$, we have $$
-2h(\lambda)(\lambda k-2k+1)-k(\lambda-1)<2(\lambda k-2k+1)-k(\lambda-1)<-k+2<0,$$
which implies $f'(\lambda)-g'(\lambda)<0$. Since $f(1)-g(1)=1$ and $f(2)-g(2)=0$, the equation $f(\lambda)-g(\lambda)=0$ has no real root in $(1,2)$. If $\lambda \in (2,3)$, denote $$r(\lambda)= -2h(\lambda)(\lambda k-2k+1)-k(\lambda-1).$$ We have
 $$r'(\lambda)=-2(\lambda-1)^{k-2}[(\lambda-2)k+1]^2+2k(\lambda-2)(\lambda-1)^{k-1}+k.$$
  Let $t=\lambda-1$. Thus, $1<t<2$ since $\lambda\in (2,3)$, and
\begin{equation*}
  r'(t)=-2t^{k-2}[(t^2-2t+1)k^2-(t^2-3t+2)k+1]+k.
\end{equation*}
Let $$s(t)=(t^2-2t+1)k^2-(t^2-3t+2)k+1.$$ Obviously, $s(t)$ is monotonically increasing in $(1,2)$. Thus, $r'(t)=0$ has a unique root in $(1,2)$. This shows that $f(\lambda)-g(\lambda)$ is first decreasing and then increasing in $(2,3)$. Hence,  the equation $f(\lambda)-g(\lambda)=0$ has the unique root in $(2,3)$ and then (vii) holds.

(f2) If $p_2$ and $p_3$ are even, we have
\begin{equation*}
  \begin{cases}
      (\lambda-2) x^{k-1}_{k}= \pm x_{2k-1}( \sqrt{|\frac{x_k x_{2k-1}}{\lambda-1}|})^{k-2} -(\frac{x_{k}}{1-\lambda})^{k-1}, \\
    (\lambda-2) x^{k-1}_{2k-1}=\mp x_{k}(\sqrt{|\frac{x_k x_{2k-1}}{\lambda-1}|})^{k-2} -(\frac{x_{2k-1}}{1-\lambda})^{k-1},
  \end{cases}
\end{equation*}
which deduces
\begin{equation*}
  [(\lambda-2)(\lambda-1)^{k-1}-1](x^k_k-x^k_{2k-1})=0.
\end{equation*}
So, we obtain $(\lambda-2)(\lambda-1)^{k-1}-1=0$ or $x_k =\pm x_{2k-1}$.
If $x_k =\pm x_{2k-1}$, then
$$[(\lambda-2)(\lambda-1)^{k-1}-1]^2-(\lambda-1)^k=0.$$
 After a discussion similar to (f1), we show that the above equation has the unique root in $(2,3)$.
If  $(\lambda-2)(\lambda-1)^{k-1}-1=0$, then by (ii), we have $\lambda\in(0,1)$ or $\lambda\in (2,3)$. Thus, (vi) holds.

(f3) If $p_2$ is odd and $p_3$ is even or $p_2$ is even and $p_3$ is odd, we have
\begin{equation*}
  \begin{cases}
      (\lambda-2) x^{k-1}_{k}= \pm x_{2k-1}( \sqrt{|\frac{x_k x_{2k-1}}{\lambda-1}|})^{k-2} \mp (\frac{x_{k}}{\lambda-1})^{k-1}, \\
    (\lambda-2) x^{k-1}_{2k-1}=\pm x_{k}(\sqrt{|\frac{x_k x_{2k-1}}{\lambda-1}|})^{k-2} \mp (\frac{x_{2k-1}}{1-\lambda})^{k-1},
  \end{cases}
\end{equation*}
or
\begin{equation*}
  \begin{cases}
    [(\lambda-2)(\lambda-1)^{k-1}\pm1]^2 x^{2k}_{k}= x^{k}_{k}x^{k}_{2k-1}(\lambda-1)^{k}, \\
    [(\lambda-2)(1-\lambda)^{k-1}\pm1]^2 x^{2k}_{2k-1}= x^{k}_{k}x^{k}_{2k-1}(1-\lambda)^{k}.
  \end{cases}
\end{equation*}
Hence,
\begin{equation*}
  (\lambda-2)(\lambda-1)^{k-1}(x^{2k}_{k}-x^{2k}_{2k-1})=0,
\end{equation*}
which implies $\lambda=2$ or $x_k =\pm x_{2k-1}$. We assume $x_k =\pm x_{2k-1}$, then
\begin{equation*}
    |(\lambda-2)(\lambda-1)^{k-1}\pm1|=|(\lambda-2)(1-\lambda)^{k-1}\pm1|.
\end{equation*}
This shows that $\lambda=2$. Thus, (i) holds.
\end{proof}

\Section{Numerical experiments and convergence analysis}

We now give some numerical experiments to show the conclusions given in Theorems \ref{l1em2.1} and \ref{l1em2.3}.  We apply the bisection method to solve the real root of the polynomial equations given in Theorems \ref{l1em2.1} and \ref{l1em2.3}. For fixed $k$, the computational complexity of this method is
$O(\log\frac{1}{\epsilon})$, where $\varepsilon$ is a given tolerance. We also investigate the changing trend of Laplacian  H-spectrum of $G_{k,3}$ as $k$ increasing.

By making a lot of experiments with $3\leq k\leq 50$, we find  some good properties of  Laplacian  H-spectrum of $G_{k,3}$. The numerical results are reported in Figure \ref{fig1}, where (a) and (b) show the distribution of all H-eigenvalues of Laplacian tensors for odd/even-uniform loose paths with length three, respectively.
\begin{figure}[!htb]
\centering
\subfigure[]{\includegraphics[height=2.5in,width=3in]{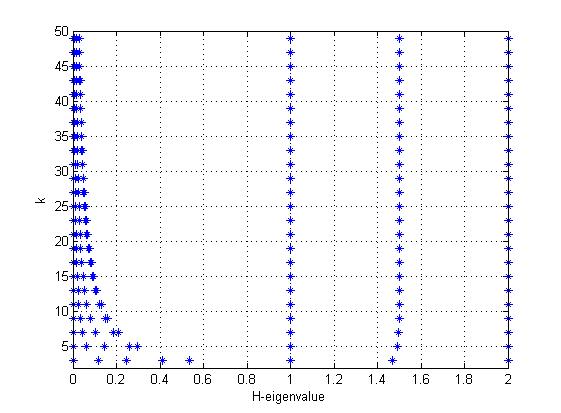}}
\subfigure[]{\includegraphics[height=2.5in,width=3in]{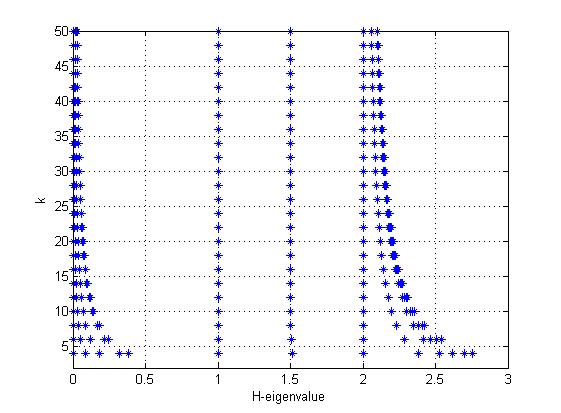}}
\caption{{\small Distribution of Laplacian H-spectrum of $G_{k,3}$. (a) $k$ is odd; (b) $k$ is even.}}
 \label{fig1}
\end{figure}
From Figure \ref{fig1}, we see that Laplacian  H-spectrum of $G_{k,3}$ converges to $\{0,1,1.5,2\}$ when $k$ goes to infinity.
\begin{Theorem}\label{mythm}{\rm
Let $\mathcal{L}_k$ be Laplacian tensor of $G_{k,3}$ and $\lambda_k$ be its an H-eigenvalue. Then $\lim\limits_{k\to\infty}\lambda_k=\lambda^*$, where $\lambda^*\in \{0,1,1.5,2\}$.}
\end{Theorem}
Yuan, Qi and Shao \cite{yqslimit} showed that $\{ \lambda(\mathcal{L}_k)\}$ is a strictly decreasing sequence. Clearly, the conclusion in Theorem \ref{mythm} is much stronger than theirs. To show Theorem \ref{mythm}, we need the following lemmas.
\begin{Lemma}\label{lem1}{\rm
Let $k$ be odd and $\lambda_k$ be the real root of $(\lambda-2)(1-\lambda)^{k-1}+1=0$ in $(0,1)$. Then $\{\lambda_k\}$ is a strictly decreasing sequence and $\lim\limits_{k\to\infty}\lambda_k=0$.}
\end{Lemma}
\begin{proof} Let $f(\lambda)=(\lambda-2)(1-\lambda)^{k-1}+1$. For fixed $k$, by Theorem \ref{l1em2.1} (ii), $\lambda_k$ is the unique real root of $f(\lambda)=0$ in $(0,1)$.
Let $t_k=1-\lambda_k$. Then $t_k\in (0,1)$ and
$$f(\lambda_k)=0\quad \Leftrightarrow \quad k=-\frac{\ln{(t_k+1)}}{\ln{t_k}}+1.$$
It is easy to see that  $g(t):=-\frac{\ln{(t+1)}}{\ln{t}}$ is a strictly increasing function in $(0,1)$. Hence, we obtain
$$k_1<k_2 \quad \Leftrightarrow \quad t_{k_1}<t_{k_2} \quad \Leftrightarrow \quad \lambda_{k_1}>\lambda_{k_2}, $$
which implies that $\{\lambda_k\}$ is a strictly decreasing sequence. Therefore,  the limit of $\{\lambda_k\}$  must exist.
Let $\lim\limits_{k\to\infty}\lambda_k=\lambda^*$. Then $0\le \lambda^*\le 1$.  It follows from $\lim \limits_{k\rightarrow \infty}(1-\lambda_k)=\lim \limits_{k\rightarrow\infty}(\frac{1}{2-\lambda_k})^{\frac{1}{k-1}}$ that $\lambda^*=0$.
\end{proof}
\begin{Lemma}\label{lem2}{\rm
Let $k$ be even and $\lambda_k$ be the real root of $(\lambda-2)(\lambda-1)^{k-1}-1=0$ in $(2,3)$. Then $\{\lambda_k\}$ is a strictly decreasing sequence and $\lim\limits_{k\to\infty}\lambda_k=2$.}
\end{Lemma}
\begin{proof} Let $f(\lambda)=(\lambda-2)(\lambda-1)^{k-1}-1$. For fixed $k$, by Theorem \ref{l1em2.3} (ii), $\lambda_k$ is the unique real root of $f(\lambda)=0$ in $(2,3)$.
Let $t_k=\lambda_k-1$. Then $t_k\in (1,2)$ and
$$f(\lambda_k)=0\quad \Leftrightarrow \quad  k=-\frac{\ln{(t_k-1)}}{\ln{t_k}}+1.$$
It is easy to see that $g(t):=-\frac{\ln{(t-1)}}{\ln{t}}$ is a strictly decreasing function in $(1,2)$. Therefore,
$$k_1<k_2 \quad \Leftrightarrow \quad t_{k_1}>t_{k_2} \quad \Leftrightarrow \quad \lambda_{k_1}>\lambda_{k_2}, $$
which implies that $\{\lambda_k\}$ is a strictly decreasing sequence. Hence, the limit of $\{\lambda_k\}$  must exist.
Let $\lim\limits_{k\to\infty}\lambda_k=\lambda^*$. Then $2\le \lambda^*\le 3$. Moreover, $\lambda^*=2$. In fact, if $\lambda^*\neq 2$, then by
$\lim \limits_{k\rightarrow \infty}(\lambda_k-1)=\lim \limits_{k\rightarrow\infty}(\frac{1}{\lambda_k-2})^{\frac{1}{k-1}}$, we must have $\lambda^*-1=1$ and hence $\lambda^*=2$. This is a contradiction.  Thus, we complete the proof.
\end{proof}
\begin{Lemma}\label{lem7}{\rm
Let $k$ be odd and $\lambda_k$ be the real root of $(\lambda-2)^2(1-\lambda)^{k-1}+2\lambda-3=0$ in $(1,\frac{2k}{1+k})$. Then $\{\lambda_k\}$ is a strictly increasing sequence and $\lim\limits_{k\to\infty}\lambda_k=1.5$.}
\end{Lemma}
\begin{proof} Let $f(\lambda)=(\lambda-2)^2(1-\lambda)^{k-1}+2\lambda-3$. For fixed $k$, by Theorem \ref{l1em2.3} (iv), $\lambda_k$ is the unique real root of $f(\lambda)=0$ in $(1,\frac{2k}{1+k})$.
Let $t_k=\lambda_k-1$. Then
$1-2t_k=(t_k-1)^2t_k^{k-1}$. Since $k$ is odd, we have $0<t_k<0.5$. Furthermore,
$$k=\frac{\ln{(1-2t_k)}-2\ln{(1-t_k)}}{\ln{t_k}}+1.$$
It is easy to see that $g(t):=\frac{\ln{(1-2t)}-2\ln{(1-t)}}{\ln{t}}$ is a monotonically increasing function in $(0,0.5)$ by analyzing its derivative. Therefore,
$$k_1<k_2 \quad \Leftrightarrow \quad t_{k_1}<t_{k_2} \quad \Leftrightarrow \quad \lambda_{k_1}<\lambda_{k_2}, $$
which implies that $\{\lambda_k\}$ is a strictly increasing sequence. Hence, the limit of $\{\lambda_k\}$  must exist.
Let $\lim\limits_{k\to\infty}\lambda_k=\lambda^*$. Then $1< \lambda^*\le 1.5$ due to $0<t_k<0.5$. If $1< \lambda^*<1.5$, then by computing
$\lim \limits_{k\rightarrow \infty}(\lambda_k-1)=\lim \limits_{k\rightarrow\infty}(\frac{3-2\lambda_k}{(2-\lambda_k)^2})^{\frac{1}{k-1}}$,  we have $\lambda^*-1=1\Leftrightarrow \lambda^*=2$. This is a contradiction, and hence $\lambda^*=1.5$. We complete the proof.
\end{proof}

Using the same techniques as the proof of Lemmas \ref{lem1}, \ref{lem2} and \ref{lem7}, we can easily obtain the other lemmas.
\begin{Lemma}\label{lem3}{\rm
Let $k$ be odd and $\lambda_k$ be the real root of $(\lambda-2)^2(1-\lambda)^{k-2}-1=0$ in $(0,1)$. Then $\{\lambda_k\}$ is a strictly decreasing sequence and $\lim\limits_{k\to\infty}\lambda_k=0$.}
\end{Lemma}
\begin{Lemma}\label{lem5}{\rm
Let $k$ be odd and $\lambda_k$ be the real root of $(\lambda-2)^2(1-\lambda)^{k-1}+2\lambda-3=0$ in $(0,1)$. Then $\{\lambda_k\}$ is a strictly decreasing sequence and $\lim\limits_{k\to\infty}\lambda_k=0$.}
\end{Lemma}
\begin{Lemma}\label{lem4}{\rm
Let $k$ be even and $\lambda_k$ be the real root of  $(\lambda-2)^2(\lambda-1)^{k-2}-1=0$ in $(2,3)$. Then $\{\lambda_k\}$ is a strictly decreasing sequence and $\lim\limits_{k\to\infty}\lambda_k=2$.}
\end{Lemma}
\begin{Lemma}\label{lem6}{\rm
Let $k$ be even and $\lambda_k$ be the real root of $(\lambda-2)^2(\lambda-1)^{k-1}-2\lambda+3=0$ in $(2,3)$. Then $\{\lambda_k\}$ is a strictly decreasing sequence and $\lim\limits_{k\to\infty}\lambda_k=2$.}
\end{Lemma}
\begin{Lemma}\label{lem8}{\rm
Let $k$ be even and $\lambda_k$ be the real root of $(\lambda-2)^2(\lambda-1)^{k-1}-2\lambda+3=0$ in $(1,\frac{2k}{1+k})$.  Then $\{\lambda_k\}$ is a strictly decreasing sequence and $\lim\limits_{k\to\infty}\lambda_k=1.5$.}
\end{Lemma}

Combining the conclusions in Lemmas \ref{lem1}-\ref{lem8} and  Theorems \ref{l1em2.1} and \ref{l1em2.3}, we see that we can not prove Theorem \ref{mythm}. So, we give the following conjecture.
\begin{Conjecture}\label{lem9}{\rm
Theorem \ref{mythm} is true for the $k$-uniform loose path of length $3$.}
\end{Conjecture}

\section{Some concluding remarks}
Motivated by the applications of Laplacian H-spectrum in edge cut and edge connectivity of a $k$-uniform hypergraph and in image processing, we studied Laplacian H-spectrum of the $k$-uniform loose path of length three. By Theorems \ref{l1em2.1} and \ref{l1em2.3}, we found out all Laplacian H-eigenvalues of the loose path of length three. We showed that the number of Laplacian H-eigenvalues of the odd-uniform loose path of length three is $7$, and the number of Laplacian H-eigenvalues of the even-uniform loose path of length three is $14$. Some numerical results are given to show the efficiency of our method.  Especially, the numerical results show that its Laplacian H-spectrum  converges to $\{0,1,1.5,2\}$ when $k$ goes to infinity. Finally, we established convergence analysis for a part of the conclusion. However, we can not prove the whole conclusion. We presented a conjecture for the future research.

\end{document}